\documentclass[11pt]{amsart}
\usepackage{latexsym}
\usepackage{amsfonts}
\usepackage{color}

\newcommand{\field}[1]{\mathbb{#1}}
\newcommand{\C}{\field{C}}

\newcommand{\K}{\field{K}}
\newcommand{\N}{\field{N}}

\newcommand{\PP}{\field{P}}
\newcommand{\R}{\field{R}}

\newtheorem{defi}{Definition}[section]
\newtheorem{ex}[defi]{Example}
\newtheorem{lem}[defi]{Lemma}
\newtheorem{theo}[defi]{Theorem}

\newtheorem{pr}[defi]{Proposition}
\newtheorem{re}[defi]{Remark}

\font\tenmsy=msbm10

\def\Bbb#1{\hbox{\tenmsy#1}} 
\title[Asymptotic critical values ]{Reaching  generalized critical values of
a polynomial} \makeatletter

\@addtoreset{equation}{section}
\makeatother
\author{Zbigniew Jelonek and  Krzysztof Kurdyka}
\address[Z. Jelonek]{Instytut Matematyczny\\
Polska Akademia Nauk\\
\'Sniadeckich 8, 00-956 Warszawa,
Poland }
\email{najelone@cyf-kr.edu.pl}
\address[K. Kurdyka ] {LAMA Laboratoire de Math\'ematiques, UMR 5127
Uni\-ver\-si\-t\'e de Savoie, Campus scientifique, F-73376 Le Bourget
du Lac cedex, France}
\email{kurdyka@univ-savoie.fr}

\thanks{The first author was partially supported by  Universit\'e de Savoie  and by NCN(Poland), 2014-2017. The second author was   partially supported by
ANR(France) grant STAAVF}

\date {\today}

\begin{document}
\begin{abstract}
Let $f: \K^n \to \K$ be a polynomial, $\K=\R, \,\C$.  We give an algorithm to compute the set of generalized
critical values. The algorithm uses a finite dimensional space of  rational arcs along which we can reach  all generalized critical values of $f$.

\end{abstract}

\bibliographystyle{alpha}

\maketitle

\section{Introduction.}

Let $f:\K^n \to \K $ be a polynomial ($\K=\R$ or $\K=\C$). Over
forty  years ago R. Thom \cite{thom}  proved that $f$ is a $C^\infty$-fibration
outside a finite subset of the target; the smallest such a set is called {\it the
bifurcation set of} $f$, we  denote it by $B(f)$. In a natural
way appears a fundamental question: how to determine the set
$B(f)$ ?

Let us recall that in general
the set $B(f)$ is bigger than $K_0(f)$, the set
of critical values of $f$. It also contains  the set $B_\infty (f)$
of bifurcations points at infinity. Roughly speaking, 
$B_\infty (f)$ consists of points at which $f$ is not a locally trivial
fibration at infinity (i.e., outside a large ball). For $n=2$ in the complex case the set  $B_\infty (f)$ 
can be  effectively  computed (see e.g., \cite{HL}, \cite{Su}). 
In the real  case  and $n=2$ the set  $B_\infty (f)$   has an explicit description, see \cite{CP}.
For $n\ge3$ the computation of $B_\infty (f)$ is a challenging  open problem even for $\K=\C$.
To control the set $B_\infty (f)$ one can use the set of {\it
asymptotic critical values of} $f$
$$K_\infty (f)=\{ y \in \K :\exists {
x_{\nu} \in \K^n,  \, x_{\nu}\rightarrow\infty} \ s.t. \ f(x_{\nu})\rightarrow y, \, \Vert
x_{\nu}\Vert \Vert d f(x_{\nu})\Vert\rightarrow 0\}.$$ If $c\notin
K_\infty (f)$, then it is  usually  said that $f$ satisfies {\it
Malgrange's condition}  at $c$. It is known (\cite{par},
\cite{par2}) that $B_\infty (f)\subset K_\infty (f)$. We call
$K(f)=K_0(f)\cup K_\infty (f)$ the  {\it 
 set of generalized crtitical values of } $f$. Thus  in general
$B(f)\subset K(f)$. In the case $\K=\C$ we gave  in \cite{j-k} an algorithm to
compute the set $K(f)$.

 In the
 real case, that is,  for  a given real polynomial $f:\R^n \to \R$ we can compute  $K (f_\C)$,  the set of generalized critical values of the complexification  $f_\C$  of $f$.
  However in general   the set $K_\infty(f)$  of  asymptotic critical values of $f$ may be  smaller than
  $\R\cap K _\infty(f_\C)$. Precisely, it  is possible  (see Example
\ref{ex1}) that there exist a sequence $x_{\nu}\in \C^n \setminus\R^n$ with 
  $\Vert x_{\nu} \Vert \to \infty$ such that $ f(x_{\nu})\rightarrow y\in \R$ and $\Vert
x_{\nu}\Vert \Vert d f(x_{\nu})\Vert\rightarrow 0$, but there is no  sequence $x_{\nu} \in \R^n$ with this property.
To our best knowledge
 no method was known  to detect  algorithmically this situation.

In the paper we propose another approach to the  computation of  generalized  critical values which works both
in  the complex and in the real case. The main new idea is to use a finite dimensional space of  rational arcs along which we can reach  all asymptotic critical values.

Asymptotic and generalized critical values appear for instance  in the problem of optimization of real  polynomials
(see e.g.,  papers of H\`a  and Pham \cite{huj1}, \cite{huj2}).  In fact,  as they observed,  if a polynomial $f$ is bounded from below then $\inf f \in K(f)$. Numerical and complexity  aspects of this approach to the optimization
of polynomials were studied recently by M. Safey El Din \cite{safey1}, \cite{safey2}.

{\it Acknowledgments.} We thank the referees   for their careful and patient  reading of the first  version of our manuscript.

\section{Preliminaries}

\indent
   Here we assume that $\K=\C.$ If $f: X \to Y$ is a dominant, generically finite polynomial  map of smooth affine varieties   and  $\{ x \}$  is  an  isolated
component  of  the  fiber $f^{-1}  (f(x))$, then  we  denote
the multiplicity of  $f$ at  $x$ by  ${\rm mult}_x (f)$
and the number of points in a generic fiber of $f$ by $\mu (f)$.

Let $X, Y$ be  affine varieties.  Recall that a  mapping $F:X\rightarrow Y$ is {\it not
proper} at  a point $y \in Y$ if there  is  no  neighborhood
$U$  of  $y$  such  that $F^{-1} (\overline{U})$ is compact. In
other words, $F$ is  not proper at $y$ if there is  a sequence
$x_\nu \rightarrow\infty $ such that $F(x_\nu)\rightarrow y $. Let
$S_F$ denote the set of points at which the mapping $F$ is  not
proper.  The   set $S_F$ has the following properties
(see \cite{jel}, \cite{jel1}, \cite{jel2}):

\begin{theo} \label{setSF} Let $X\subset \C^k$ be an irreducible variety of dimension $n$ and let
$F = (F_1,\dots , F_m): X \rightarrow \C^m$ be a generically finite
polynomial  mapping. Then  the  set $S_F$ is an algebraic subset
of $\C^m$ and it is either empty or of  pure dimension $n-1$.
Moreover, if $n=m$ then   $${\rm deg} \  S_F \leq  \frac{
D(\prod^n_{i=1} {\rm deg} \ F_i) - \mu (F)} {\min_{1\le i \le
n}\deg\ F_i},$$ where $D=\deg X$ and $\mu (F)$ denotes the
geometric degree of $F$ (i.e., the number of points in a
generic fiber of $F$).
\end{theo}

In the case of a polynomial map of normal affine varieties it  is
easy to  show  the following.

\begin{pr}
Let $f: X \rightarrow Y$ be a dominant polynomial  map of normal
affine varieties. Then $f$ is proper  at  $y \in Y$   if   and
only    if $f^{-1}  ({y})  = \{ x_1,\ldots, x_r \}$ is a finite
set  and $\sum_{i=1}^{r} {\rm mult}_{x_i}  (f) = \mu (f)$.
\end{pr}
Actually  we prove a stronger result.

\begin{pr}\label{lemat}
Let $f: X \rightarrow Y$ be a dominant polynomial  map of normal
affine varieties. Assume that points $x_1,\dots,x_r$ are isolated
components of the fiber $f^{-1}  ({y}).$ If  $\sum_{i=1}^{r} {\rm
mult}_{x_i}  (f) = \mu (f)$, then $f$ is proper at $y$ and $f^{-1}
({y})=\{x_1,\dots,x_r\}.$
\end{pr}

\begin{proof}
Let $\overline{X}$ be the projective completion of $X.$ Let $V$ be  the normalization of the closure of  the graph $f$ in $\overline{X}\times Y$.
Then $X\subset V$ and there is a proper mapping $\overline{f}: V\to Y$ such that $\overline{f}_{|X}=f.$
By   the Stein Factorization Theorem (see \cite{iit}, pp. 141-142), there exists a normal variety $W$, a finite morphism
$h : W\rightarrow Y$, and a surjective morphism $g: V\rightarrow W$
with connected fibers, such that $\overline{f}=h\circ g.$ In particular the mapping $g$ is birational.
Let $w_i=g(x_i).$ By  the Zariski Main Theorem the mapping $g$ is a local biholomorphism near $x_i$. This means
that ${\rm mult}_{x_i} \overline{f}={\rm mult}_{x_i} f= {\rm mult}_{w_i} h.$ Since the mapping $h$ is finite and $\mu(h)=\mu(f)$, we have 
$h^{-1}(y)=\{ w_1,\dots, w_r\}.$  Consequently, $\overline{f}^{-1}(y)=\{ x_1,\dots, x_r\},$ which implies that $f$ is proper at $y.$
\end{proof}
The next fact will be useful.

\begin{pr}\label{proper}
Let $X$ be a normal affine variety of dimension $n$ and let $f : X\to\C^n$ be a dominant mapping.
Let $b\in \C^n$. If there exists a relatively compact set  $K\subset X$  and  a sequence of points $b_i\in\C^n$ such that

1)  $f^{-1}(b)\subset K$ and $f^{-1}(b_i)\subset K$ for $i \in\N $,

2) $\# f^{-1}(b_i)=\mu(f)$,

3) $\lim_{i\to\infty} b_i=b$,

\noindent then $f$ is proper at $b$.
\end{pr}

\begin{proof}
Since $X$ is locally compact we can assume that the set $K$ is open. Let $U_n\subset Y$ be a ball of radius $1/n$ with center at $b.$
Assume that for any  $n$ the mapping $f: K\cap f^{-1}(U_n)\to U_n$ is not proper. Since the mapping $f: \overline{K}\cap f^{-1}(U_n)\to U_n$
is proper, we have $\overline{K}\cap f^{-1}(U_n)\not={K}\cap f^{-1}(U_n)$. Take $x_n\in \overline{K}\cap f^{-1}(U_n)\setminus {K}\cap f^{-1}(U_n)$.
Since $\overline{K}$ is compact, the sequence $x_n$ has an accumulation point $x\in \overline{K}\setminus K.$ But $f(x)=y$, which contradicts the fact that
$f^{-1}(y)\subset K.$

Consequently, the mapping  $f: K\cap f^{-1}(U_n)\to U_n$ is  proper for some $n.$ Thus it is an analytic cover (see Theorem 21, p. 108 in \cite{gun-ros}).
Moreover, it is easily  seen that this cover is $\mu(f)$-sheeted. In particular for any $y\in U_n$ we have $\sum_{f(x)=y, \ x\in S} {\rm mult}_{x}  (f) =
\mu (f)$. To conclude we apply   Proposition \ref{lemat}.
\end{proof}

\section{Parametrizing branches at infinity }

We begin  with a variant of  the Puiseux Theorem.

\begin{lem}\label{puisex}
Let $C\subset \C^n$ be an algebraic  curve of degree $d$. Assume that
$a\in (\Bbb P^n\setminus \C^n)\cap \overline{C}$ is a point at
infinity of $C$. Let $\Gamma$ be an irreducible  component of the
germ ${\ \overline{C}_a}$. Then there exist an integer $s\le
d$ and a real number $R>0$ such that  $\Gamma$ has a
holomorphic parametrization of the type
$$x=\sum_{-\infty\le i\le s} a_it^i,\,\, |t|>R,$$ where $  t\in \C,\, a_i\in
\C^n$ and $\sum_{i>0} |a_i|>0.$
\end{lem}

Note that  $\overline{C}$,  the closure  of $C$ in $\PP^n$, is the same for the strong topology and the Zariski topology, since we work over $\C$. So in fact  
$\overline{C}$ is an algebraic set in $\PP^n$.
\begin{proof}
First choose an
affine system of coordinates in $\C^n$ in a generic way.  Let
$a=(0:a_1:\dots :a_n)\in \Bbb P^n.$ We can assume that
$a_1\not=0$.
Now choose 
coordinates in the affine chart $U_1= \Bbb P^n\setminus H$, where
$H=\{ x : x_1=0\}.$ Take $y_1=x_0/x_1$ and $y_i=x_i/x_1$ for
$i=2,...,n.$ Put $L=\{ x\in H: x_0=0 \}.$ By our assumption we
have $L\cap \overline{C}=\emptyset.$ In particular the projection
$\pi_L: \overline{C}\setminus H\ni (y_1,\dots ,y_n)\mapsto y_1\in \C$
is finite.  So there is a punctured disc $U_\delta^*=\{ z\in \C
: 0<|z|< \delta\}$ such that the mapping $$\rho :\Gamma \cap
\pi_L^{-1}(U_\delta^*)\ni x\mapsto \pi_L(x)\in U_\delta^*$$ is proper. We can also
assume that the set $\Gamma^\prime:=\Gamma \cap \pi_L^{-1}(U_\delta^*)$ is connected 
smooth and $\rho$ has no critical values on $U_\delta^*$. In particular
$\rho$ is a holomorphic covering of degree $s\le d.$

In particular $\rho^{-1} : U_\delta^*\ni z\mapsto (z,
h_2(z),\dots ,h_n(z))\in \Gamma^\prime$ is an $s$-valued holomorphic
function. If we compose it with  $z\mapsto z^s$ we obtain a
holomorphic bounded  function on $ U_{\delta^{1/s}}^*$. 
By Riemann's theorem on removable singularities this  function  extends to
a holomorphic function on the disc $U_{\delta^{1/s}}= \{ z\in \C
: \, |z|< \delta^{1/s}\}$. 
{
 Consequently, the mapping $$z\mapsto (z^s,
h_2(z^s),\dots,h_n(z^s))=(z^s, g_2(z),\dots, g_n(z))$$
is holomorphic in  $U_{\delta^{1/s}}$.
(Precisely we have $s$ such functions which are  of the form
$(z^s, g_2(\xi z),\dots, g_n(\xi z))$, where $\xi^s=1$.)
Hence in the original coordinates  we have
$$z\mapsto (1/z^{s},
g_2(z)/z^s,\dots ,g_n(z)/z^s).$$
Now put $z=t^{-1}$, where $|t|> R:=1/\delta.$
So
$$t\mapsto (t^{s},
t^sg_2(t^{-1}),\dots,t^sg_n(t^{-1}))$$
is the desired parametrization of $\Gamma$.
}
\end{proof}

\vskip 1cm


To state the real version let us recall some well known facts (see  e.g., \cite{kurdyka1}, Preliminaries).
Consider the  standard  embedding of  $\R^n$ in its projective closure denoted by $\PP^n(\R)$.
Let $C\subset \R^n$ be  an algebraic curve, denote by ${\overline C}^Z$ the  Zariski  closure  of $C$   in  $\PP^n(\R)$ and 
by ${\overline C} $ the closure  of $C$ in  $\PP(\R^n)$ for the strong topology. 
Clearly   ${\overline C \subset \overline C}^Z$ and $ \overline{ C}^Z\setminus \overline  C$ is a finite (possibly empty) set.
Hence, if  $a\in (\Bbb P^n(\R)\setminus \R^n)\cap \overline{C}$,  then we have the equality of germs ${\overline C_a= \overline C}^Z_a$.
Finally, recall that  $\Gamma$ is an irreducible component of the germ  of a real algebraic  curve $A$ at $a\in \PP^n$, which is a non-isolated  point of $A$, if and only if 
  there exists a real analytic injective parametrization $\gamma: (-\delta, \delta) \to A\subset  \PP^n(\R)$ of $\Gamma$.
  By the degree of a real algebraic curve we mean the degree of its complexification.
  

\begin{lem}\label{puisexreal}
Let $C\subset \R^n$ be a real algebraic  curve of degree $d$. Assume that
$a\in (\Bbb P^n(\R)\setminus \R^n)\cap \overline{C}$ is a point at
infinity of $C$. Let $\Gamma$ be an irreducible  component of the
germ ${\ \overline{C}_a}$. Then there is an integer $s\le
d$ and a real number $R>0$ such that  $\Gamma$ has a
real analytic  parametrization of the type
$$x=\sum_{-\infty\le i\le s} a_it^i,\,\, |t|>R,$$ where $  t\in \R,\, a_i\in
\R^n$ and $\sum_{i>0} |a_i|>0.$
\end{lem}
\begin{proof} We shall explain how to adapt  the proof of Lemma \ref{puisex} to the real case. Denote by $C^\prime \subset \PP^n$ the complexification of the real  curve 
$\overline C^Z$.
Hence the complex conjugation $\sigma$ acts as an involution on $C^\prime $, moreover $\overline C^Z$ is the set of fixed points of $\sigma$. 
Note that in the proof of  Lemma \ref{puisex} we can choose real coordinates, so the action of $\sigma$ will be preserved. 
We shall use  the action of $\sigma$ on the germ $C^\prime _a$.  
By the assumption there exists a branch of $C^\prime _a$ which contains the real branch $\Gamma $; we call this branch $\Gamma^\prime $. 
So $\sigma (\Gamma^\prime ) = \Gamma^\prime $, which means that the graph of the multivalued function 
$h= (h_2,\dots,h_n)$ is invariant under the action of $\sigma$. The set $\Gamma_s : = \{(z^s,h(z^s)):\, z \in U^*_{\delta^{1/s}}\}$ is a disjoint union 
of graphs of $s$ holomorphic functions  in the  punctured  disc  $ U_{\delta^{1/s}}^*$.

Note that $\sigma (\Gamma_s ) = \Gamma_s $. Since one of those graphs contains real points (coming from the real branch $\Gamma$), this graph is stable under $\sigma$.  We call this function $g=(g_2,\dots,g_n)$.
Therefore
$g_j(\bar z) =\overline {g_j(z)}$ for $j=2,\dots,n$, 
which means  that all coefficients  in the power series expansion (at $0\in \C$) of  $g_j $ are real. 
Hence the lemma follows.

\end{proof}
The following elementary but useful  lemma can be checked by direct computations.

\begin{lem}\label{truncation}
Assume that a holomorphic  curve has  a parametrization of the type
$$x(t)=\sum_{-\infty\le i\le s} a_it^i,\,\, |t|>R,$$ where $  t\in \C,\, a_i\in \C^n, \sum_{i>0} |a_i|>0$ and $s\ge 0$ is an integer.
Let $f:\C^n \to \C$ be a polynomial of degree $d$. Set
$$\tilde x(t)=\sum_{-(d-1)s\le i\le s} a_it^i,\,\, |t|>R.$$  Assume that $ \lim_{t\to \infty} f( x(t)) = b \in \C$. Then
$$\lim_{t\to \infty} f(\tilde x(t))=\lim_{t\to \infty} f( x(t)).$$
The same statement holds in the real case.
\end{lem}

\section{The complex case}
Let us recall  a well-known  fact.

\begin{lem}\label{projekcja}
For sufficiently general numbers $a_{ij}\in \C$ the mapping
$$\pi: X\ni (x_1,\dots,x_m)\mapsto \Big(\sum_{j=1}^m a_{1j}x_j,
\sum_{j=2}^m a_{2j}x_j,\dots,\sum_{j=n}^m a_{nj}x_j\Big)\in \Bbb
C^n$$ is finite.
\end{lem}

We can now state an  effective variant of  the  curve selection lemma.

\begin{theo}\label{luk}
Let $F:\C^n\ni x\mapsto (f_1(x),\dots ,f_m(x))\in \C^m$ be a generically
finite polynomial mapping. Assume that $\deg f_i=d_i$ and $d_1\ge
\cdots \ge d_m.$ Let $b\in \C^m$ be a point at which 
$F$ is not proper. Then there exists a rational curve  with a
parametrization of the form
$$x(t)=\sum_{-(d-1)D-1\le i\le D} a_it^i, \ t\in \C^*,$$ where $a_i\in
\C^n, \sum_{i>0} |a_i|>0$ and $D=\prod^{n}_{i=2} d_i, \ d=d_1,$
such that
$$\lim_{t\to \infty} F(x(t))=b.$$
\end{theo}

\begin{proof}
First we consider the case $m=n.$  We have two possibilities:

1) the fiber $F^{-1}(b)$ is finite,

2) the fiber $F^{-1}(b)$ is infinite.

1) We can assume that the system of coordinates in the target is sufficiently
general. In particular we can assume that the line $L=\{ x:
x_2=b_2, x_3=b_3,\dots , x_n=b_n\}$,  where $b=(b_1,\dots,b_n)$,  is
contained neither in the set $S_F$ nor in the set of critical
values of $F$ and it omits the (constructible) set over which the mapping $F$ has infinite fibers.


Let $C=F^{-1}(L).$ By our assumptions $C$ is a curve.
By Proposition \ref{proper} we see that the preimage under $F_{|C}$ of any neighborhood of $b$ on $L$ cannot be relatively compact,
i.e., the point $b$ is a non-proper point
of the mapping $F_{|C}.$ In particular there exists a holomorphic
branch $\Gamma$ of $C$ such that $\lim_{x\in \Gamma, x\to\infty}
F(x)=b.$ Note that $\deg C \le D.$ By Lemma \ref{puisex} we can
assume that the branch $\Gamma$ has a parametrization of the form
$$x=\sum_{-\infty\le i\le D} a_it^i, \ |t|>R,\ {\rm and}\ \sum_{i>0} |a_i|>0.$$

From Lemma \ref{truncation} it follows  that
 $$F\left(\sum_{-(d-1)D-1\le i\le D} a_it^i\right)=b+\sum_{i=1}^\infty
c_i/t^i,$$
which proves the theorem in  case 1).

2) Let $W$ be an irreducible component of  $F^{-1}(b)$ of positive dimension. 
By the  general Bezout's formula  \cite[Thm. 2.2.5]{flenner}
 we see that deg $W\le D$.
Using a hyperplane section we easily obtain a curve $C\subset W$ with  $\deg C\le D$. Now we conclude  as in the in the previous case.

In the general case when $m>n$, we set  $X=\overline{F(\C^n)}$. Let $\pi: \C^m\to \C^n$ be a generic projection as in Lemma \ref{projekcja}. We can assume that
$\pi^{-1}(\pi(b))\cap S_F=\{b\}$ since $S_F$ is of dimension less than $n$. Take $G=\pi\circ F.$ Then $G=(g_1,\dots ,g_n)$ and $\deg g_i\le d_i$.

 By the first part of our proof there
exists a holomorphic
branch $\Gamma$  of a curve $C$ such that $\lim_{x\in \Gamma, x\to\infty}
G(x)=\pi(b)$ and $\deg C \le D.$ Since the mapping $\pi_{|X}$ is proper  we have 
$$b^\prime:= \lim_{x\in \Gamma, x\to\infty}
F(x)\in \pi^{-1}(\pi(b))\cap X.$$
So $b^\prime\in S_F$, hence $b^\prime=b$.
\end{proof}

\begin{defi}
By a rational arc we mean
a function of the form 
$$x(t)=\sum_{- D_2\le i\le D_1} a_it^i, \ t\in
\C^*,$$ where $a_i\in \C^n.$  If $a_i \ne 0$ for $i= D_1,  \, -D_2$ we  say that 
$x(t)$ is of bidegree $(D_1,D_2).$
\end{defi}
We  identify an  arc $x(t)$ with
its coefficients. Clearly the space of all rational arcs of bidegree at most  $(D_1,D_2)$ is isomorphic to
$\C^{n(1+D_1+D_2)}$.

\begin{defi}
Let $F=(f_1,\dots ,f_m) :\C^n\to\C^m$ be a generically finite
polynomial mapping which is not proper. Assume that $\deg
f_i=d_i$, where $d_1\ge \cdots \ge d_m.$ By  the asymptotic variety
of rational arcs of the mapping $F$ we mean the variety
$AV(F)\subset \C^{n(2+ \prod^n_{i=1} d_i )}$ which consists of
those rational arcs $x(t)$ of bidegree  at most   $(D_1,D_2)$,  where  $ D_1=  \prod^n_{i=2} d_i,\,
D_2 = 1+ (d_1-1) \prod^n_{i=2} d_i$ such  that

a) $F(x(t))=c_0+\sum_{i=1}^\infty c_i/t^i,$

b) $\sum_{i>0} \sum^n_{j=1} a_{ij} =1$, where
$a_i=(a_{i1},\dots ,a_{in}).$

\noindent
The  generalized asymptotic variety of $F$ is  the variety
$GAV(F)\subset \C^{n(2+ \prod^n_{i=1} d_i )}$ defined only by  condition a).
\end{defi}

\begin{re}
{\rm Condition b) ensures that 
$\lim_{|t|\to \infty} \Vert x(t)\Vert= \infty$.}
\end{re}

Notice  that $AV(F)$ and $GAV(F)$ are algebraic  subsets of
$\C^{n(2+ \prod^n_{i=1} d_i )}$. 
Recall that we  identify an  arc $x(t)$ with its coefficients $a = (a_{ij})\in \C^{n(2+ \prod^n_{i=1} d_i )}$.
If $x(t)\in GAV(F)$
then $F(x(t))= \sum_{i=0}^\infty c_i(a)/t^i$. Clearly each $c_i$ is a polynomial in $a$. The
function $c_0: AV(F)\to \C^m$ plays an important role.

\begin{pr}
$c_0(AV(F))=S_F.$
\end{pr}

\begin{proof}
Let $x(t)=\sum a_it^i\in AV(F).$ Then
$F(x(t))=c_0(a)+\sum_{i=1}^\infty c_i(a)/t^i,$ which  implies that
$c_0(a)\in S_F.$ Conversely, let $b\in S_F.$ By Theorem \ref{luk}
we can find a rational arc $x(t)=\sum_{i=-(d-1)D-1}^D a_it^i$ such
that $\lim_{t\to \infty} F(x(t))=b.$ Now change the
parametrization of $x(t), \ t\mapsto \lambda t$, so that
$\sum_{i>0} \sum^n_{j=1} \lambda^i a_{ij} =1.$ The new arc
$x'(t):=x(\lambda t)$ belongs to $AV(F)$ and $c_0(x'(t))=b.$
\end{proof}

Now let $f\in \C[x_1,\dots ,x_n]$ be a polynomial. Let us define a
polynomial mapping $\Phi : \C^n \to \C \times \C^{N}$  by
$$\Phi = \left(f,\frac{\partial f}{\partial x_1},\dots,
\frac{\partial f}{\partial x_n}, h_{11} ,h_{12},\dots ,h_{nn}\right),
$$
where $h_{ij}=x_i\frac{\partial f}{\partial x_j},\,
i=1,\dots,n,\,j=1,\dots,n .$

\begin{defi}
Let $\Phi$ be as above. Consider the mapping $c_0: AV(\Phi)\to
\C^{N}$ and the line $L:=\C\times \{ (0,\dots,0)\}\subset \C\times
\C^N$. By the  bifurcation variety  we mean the variety
$$BV(f)= \{ x(t)\in AV(\Phi): x(t)\in c_0^{-1}(L)\}.$$ Similarly, we define the generalized
bifurcation variety of rational arcs of the polynomial $f$:
$$GBV(f)= \{ x(t)\in GAV(\Phi): x(t)\in c_0^{-1}(L)\}.$$
\end{defi}

As an immediate consequence of \cite{j-k} we have:

\begin{pr}
Let $K(f)=K_0(f)\cup K_\infty(f)$ denote the set of generalized
critical values of $f.$ If we identify the line $L=\C\times \{
(0,\dots,0)\}\subset \C\times \C^N$ with $\C$, then we have
$c_0(BV(\Phi))=K_\infty(f)$ and $c_0(GBV(\Phi))=K(f).$
\end{pr}

\section{Algorithm}

In this section we give an algorithm to compute the set
$K_\infty(f)$ of asymptotic critical values as well as the set
$K(f)$ of generalized  critical values  of a complex polynomial
$f.$ Let $\deg f=d$ and $D_1=d^{n-1}, D_2=d^n-d^{n-1}+1.$

\vspace{0.5mm}

{\bf Algorithm for the set $K_\infty(f)$.}

1) Compute equations for the variety $BV(f):$

a) consider an arc $x(t)=\sum_{-D_2}^{D_1} a_i t^i\in
\C^{n(D_1+D_2+1)}$,

b) compute $f(x(t))=\sum c_i(a) t^i,$

c) compute $\frac{\partial f}{\partial x_i}(x(t))=\sum d_{ik}(a)
t^k, i=1,\dots,n,$

d) compute $\frac{\partial f}{\partial x_i}(x(t))x_j(t)=\sum
e_{ijk}(a) t^k, i,j=1,\dots,n$

e) equations for $BV(f)$ are $c_i=0$ for $i>0,$ $d_{ik}=0$ for
$k\ge 0, i=1,\dots,n$, $e_{ijk}=0$ for $k\ge 0, i,j=1,\dots,n$ and
$\sum_{i>0} \sum^n_{j=1} a_{ij} =1$, where
$a_i=(a_{i1},\dots,a_{in}).$

2) Find equations for irreducible components of
$BV(f)=\bigcup^r_{j=1} P_i$. This can be done by standard methods of
computational algebra. We can use  the MAGMA system and
a radical decomposition of  an ideal (see also
\cite{sch}).

3) Find a point $x_i\in P_i.$ This can  also be done by standard
methods. Again we  can use  the MAGMA system and the
elimination procedure. Indeed, let $P_i=V(I).$
Compute $I_k=\C[x_1,\dots,x_k]\cap I$ for $k=n,n-1,\dots$ until
$I_k=(0).$ Then take a randomly chosen  integer point,
$(a_1,\dots,a_k)$ find a zero $(a_1,\dots,a_k, b_1)$ of ideal
$I_{k+1}$ and so on.

4) $K_\infty(f)=\{c_0(x_i):\, i=1,\dots, r\}.$

\vspace{5mm} \noindent If we replace above the variety $BV(f)$ by
the variety $GAV(f)$ we get  an algorithm for computing $K(f)$. Indeed, it is
enough   to  delete   the equation $\sum_{i>0}
\sum^n_{j=1} a_{ij} =1$ in item 1e).

\section{The  real case}

We begin with a simple example.

\begin{ex} \label{ex1} For $\K= \R,\C$ consider $f_\K:\K^2 \to \K$, $f(x,y)= x(x^2 +1)^2$. Observe that
$K_\infty (f_\R) =\emptyset$. But  $0\in K_\infty (f_\C)$.
So in general
$$
K_\infty (f_\R) \ne \R \cap K_\infty (f_\C).
$$

\end{ex}

This  shows that the computation of the asymptotic critical values of  a real polynomial
cannot be reduced to the computation  of the asymptotic critical values  of its complexification.

\subsection{Effective curve selection lemma at infinity}
First we  give a construction of  a curve selection in a special  affine case.

Let  $X\subset \R^{2n}$ be an algebraic set   described by   a system  of
polynomial equations $p_i=0$,  $\deg p_i\le d$, where  $i=1, \dots, n$. Denote by  $H$ the hyperplane $\{x_1= 0\} $. Assume that
on  $Y: = X\setminus H $ the system is nondegenerate, i.e.,  $P=(p_1,\dots ,p_n): \R^{2n} \to \R^n$  is a submersion at each point of $Y$. Thus $Y $ is a smooth manifold of dimension
$n$.

\begin{pr}\label{selectpr}
Let $a\in H\cap \overline Y$. Then there exists an algebraic curve $C \subset \R^{2n}$
of degree $D \le d^n((d-1)^n +2)^{n-1}$
such that $a \in \overline {C \cap Y}$.

\end{pr}

 For simplicity we assume that $a=0$. Denote
$ \rho(x)=(\sum^{2n}_{i=1} x_i^2)^{\frac{1}{2}}$, let $S(r)$ be the sphere centered at $0$ of radius $r$, and finally  $Y(r) := Y \cap S(r)$.
The proof is a based on the following lemmas.

\begin{lem}\label{selectlem1} There exists $\varepsilon >0$ such that for any
$r\in(0,\varepsilon)$ the set
$Y(r)$ is  a smooth manifold of dimension $n-1$, in particular  is nonempty.

\end{lem}
\begin{proof}
Indeed, the function $\rho$ restricted to  the manifold  $Y$  is smooth and semialgebraic,
so it has finitely many critical values (see \cite{ris}, p. 82 or \cite{bcr}, p. 235). Let $\varepsilon_1 >0$ be the smallest
critical value (or $ \varepsilon_1  =1$ if there are no critical values).

Since $\rho |_Y: Y\to \R_+$ is  locally trivial, by Hardt's trivialization theorem (see  \cite{ris}, p. 54 or \cite{bcr}, p. 232)
there is $\varepsilon_2>0$ such that  for any $r,r'\in(0,\varepsilon_2)$ the sets
$Y(r)$ and $Y(r')$ are homeomorphic.
Since $a\in H\cap \overline Y$,   there is $r'\in(0,\varepsilon_2)$ such that
 $Y(r')$ is nonempty. Hence  $Y(r)$ is nonempty for any $r\in(0,\varepsilon_2)$.
 Finally we put  $\varepsilon=\min\{\varepsilon_1,\varepsilon_2\}$.
\end{proof}
Let us now consider the family of  functions  $g_\alpha:\R^{2n} \to \R$ of the form
$g_\alpha (x) : =x_1 \alpha (x)$, where $\alpha\in (\R^{2n})^* $ is a linear function on
$\R^{2n}$.

\begin{lem}\label{selectlem2}
For any $r\in (0,\varepsilon)$
 there exists  a proper  algebraic set  $A_r\subset (\R^{2n})^* $  such that $g_\alpha$ is a Morse function on  $Y(r)$, for any $\alpha \notin A_r$. Moreover the set $\bigcup_{r\in (0,\varepsilon)} \{r\}\times A_r$ is contained in a proper algebraic  set $A\subset \R\times(\R^{2n})^* $.
\end{lem}
\begin{proof}
The proof of the lemma uses  standard arguments  in Morse theory (see  \cite{GM}, \cite{GP}).
Recall that a function $g:Y(r) \to \R$  is  Morse if  the map $jg: Y(r)\to  T^*Y(r)$,
$jg(x) = (x,d_xg)$,  is transverse
to the zero-section $\Theta \subset T^*Y(r)$.
We have  a natural projection $\pi:Y(r)\times (\R^{2n})^*  \to T^*Y(r)$, $ \pi (x,\beta)= (x, \beta |_{T_x Y(r)})$,
where $\beta |_{T_x Y(r)}$ is the restriction of a linear form $\beta \in (\R^{2n})^* $ to the  subspace 
$T_p Y(r)$.
Consider the map
$$
\Phi :Y(r)\times (\R^{2n})^* \to  Y(r)\times (\R^{2n})^* $$
given by  $\Phi (x,\alpha) = (x, d_x g_\alpha)$.
Note that
$\Phi$ is a submersion. Indeed, the Jacobian matrix of $\Phi$ (with respect to variables in
$(\R^{2n})^* )$ is triangular, with $x_1$ on the diagonal except the entry in the left  upper corner which is $2x_1$. So this matrix is invertible since $x_1\ne 0$ for  $x\in Y(r)$.

Hence $\pi\circ \Phi :Y(r)\times (\R^{2n})^* \to T^*Y(r)$ is also a submersion. So it is transverse  to any submaniflod
of  $T^*Y(r)$, in particular to the  zero-section $\Theta \subset T^*Y(r)$. Note that $g_\alpha(x)=\pi\circ \Phi(x,\alpha)$.

By  the Transversality Theorem (see \cite{GP}, p. 68 or \cite{GM} Theorem 2.2.3, p. 53) the set $\tilde A_r$ of 
$\alpha\in (\R^{2n})^* $
such that $g_\alpha$ is  not transverse to  $\Theta$, i.e., is not a  Morse function on $Y(r)$, is nowhere dense.
Since $\tilde A_r$ is semialgebraic, its Zariski closure $ A_r$  is a proper algebraic subset of $\alpha\in (\R^{2n})^* $.

The second statement follows from the fact that set $A_r$ is defined  by polynomial equations with
$r$ as a variable parameter.
\end{proof}
\begin{lem}\label{selectlem3}
There exist  $\alpha\in (\R^{2n})^* $ and $ 0<\varepsilon' \le \varepsilon$
such that $g_\alpha$ is a Morse function on   $Y\cap S(r)$ for any $r\in(0,\varepsilon')$.

\end{lem}
\begin{proof}
 Indeed, let $A\subset \R\times(\R^{2n})^* $ by the proper algebraic  set in Lemma \ref{selectlem2}.
 Thus  there exists an affine line $\R\times {\alpha}$ which meets the set $A$ only in finitely many points. So $(0,\varepsilon')\times \{\alpha\}$ is disjoint with $A$, for some $\varepsilon' >0$ small
 enough.
 \end{proof}

Since $Y(r)$ is not compact,  a priori it is not obvious that  $g_\alpha$ has a critical point on $Y(r)$. However we have

\begin{lem}\label{selectlem4}
Assume that $Y(r)\ne \emptyset$ and that  $g_\alpha$ is  Morse on $Y(r)$. Then $g_\alpha$  has a critical point on $Y(r)$.
\end{lem}
\begin{proof}
Note that the image of $Y(r)$ under $g_\alpha$  consists of finitely many  nontrivial intervals.
Since $\overline{Y(r)}$ is compact and  $$(\overline{Y(r)}\setminus  {Y(r)}) \subset \{x_1=0\},$$
at least one of the endpoints of each of those intervals belongs to  $g_\alpha (Y(r))$.
Thus $g_\alpha$ achieves a minimum or a maximum in $Y \cap S(r)$.
\end{proof}
{\it Proof of Proposition \ref{selectpr}}. Let us fix a linear form $\alpha$ which satisfies
the requirements of 
Lemma \ref{selectlem3}.
Let   $\Xi$ be  the locus  of critical points of
$g_\alpha$ on $Y(r)$ for  $r\in(0,\varepsilon')$. Since $g_\alpha$ is a  Morse function on $Y(r)$,  for each
$r\in(0,\varepsilon')$ the set $\Xi \cap \{|x| =r\}$ is finite and nonempty by Lemma \ref {selectlem4}. Hence
$\Xi$ is a semialgebraic curve.

 The Zariski closure  of $\Xi$ is   contained in the algebraic set  given by the  equations
\begin{equation}\label{krzywaeq1}
p_1= \dots =p_n =0
\end{equation}
and
\begin{equation}\label{krzywaeq2}
dp_1\wedge \dots \wedge dp_n \wedge d\rho^2   \wedge dg_\alpha =0.
\end{equation}

Let   $\Xi_1$  be a smooth connected  component of $\Xi$ such that $a\in\overline \Xi_1$.
Then  locally $\Xi_1$ is given by  a nondegenerate system \eqref{krzywaeq1} of $n$ equations of degree at most $d$
and $n-1$ equations of degree at most $(d-1)^n +2$, which are $(n+2)\times(n+2)$ minors of
the matrix corresponding to the system \eqref{krzywaeq2}. Note that the system \eqref{krzywaeq1} with 
 \eqref{krzywaeq2} is actually nondegenerate at each point of $\Xi_1$, because critical points of a Morse function are described by a transversality condition.

Let  $C$ be the Zariski closure of  $\Xi_1$.
Hence by the general formula  of Bezout (see \cite[Thm. 2.2.5]{flenner}) the degree of the curve $C$
is at most $d^n((d-1)^n +2)^{n-1} $; note that for $d\ge 3$ we have
$d^n((d-1)^n +2)^{n-1} \le d^{n^2}$.
Hence Proposition \ref{selectpr} follows.


We can now  state a real version of Theorem \ref{luk}.

\begin{theo}\label{lukR}
Let $F:\R^n\ni x\to (f_1(x),\dots ,f_m(x))\in \R^m$ be
a  polynomial mapping. Assume that $\deg f_i\le d$.
 Let $b\in \R^m$ be a point at which the mapping
$F$ is not proper. Then there exists a rational curve  with a
parametrization of the form
$$x(t)=\sum_{-(d-1)D-1\le i\le D} a_it^i, \ t\in \R^*,$$ where $a_i\in
\R^n$, $\sum_{i>0} |a_i|>0$ and $D=(d+1)^n (d^n +2)^{n-1}$ such that
$$\lim_{t\to \infty} F(x(t))=b.$$
\end{theo}

\begin{re}
{\rm Note that in contrast to the complex case we do not assume that
$F$ is generically finite, but the bidegree of the  real rational
curve is much higher, namely $D = O(d^{n^2})$.}
\end{re}

 \begin{proof}
By Lemma \ref{projekcja} we can assume that $m=n.$ Let $\gamma(t) \in \R^n$ be a semialgebraic curve such that $\lim_{t\to \infty} |\gamma(t)|=+\infty $ and  $\lim_{t\to \infty} F(\gamma(t))=b$.  Let $\bar \gamma(t) $ be the image of $\gamma (t)$ under the  canonical
embedding $\R^n\ni (x_1, \dots,x_n)\mapsto (1:x_1:\dots:x_n) \in  \mathbb{P}^n$. Since  $\gamma$ is semi-algebraic
there exists  $a : =\lim_{t\to \infty} \bar\gamma(t) \in H_0$, where $H_0$  stands for the hyperplane at infinity.
Let us denote  by $Y$ the graph of $F$, embedded in  $  \mathbb{P}^n\times \R^n$, and by $X$ its Zariski closure  in $  \mathbb{P}^n\times \R^n$. Note that the point $( a, b)$ belongs to the closure (in the strong topology) of $Y$.
Let $\bar f_i$ stand  for the homogenization of the polynomial  $f_i$, that is,
$$
\bar f_i(x_0,x_1, \dots, x_n) = x_0^{d_i} f_i\left(\frac{x_1}{x_0}, \dots, \frac{x_n}{x_0}\right).
$$

Hence  $X\subset    \mathbb{P}^n\times \R^n$ is defined by the equations
$$
x_0^{d_i}y_i =\bar  f_i(x_0,x_1, \dots, x_n), \, i =1,\dots, n.
$$
Assume that $a_1=1$, so
$\R^n \ni(x_0,x_2, \dots, x_n)\mapsto (x_0,1,x_2, \dots, x_n)\in \mathbb P^n$ is  an affine chart around the point $a$. In this chart  $X$ is  given by the equations
$$
p_i(x_0,x_2, \dots, x_n, y_1, \dots, y_n) : = x_0^{d_i}y_i - \bar
f_i(x_0,1, x_2, \dots, x_n) =0,$$ $ \, i =1,\dots, n. $ Clearly
$\deg p_i = 1+d_i$ and $Y =X \setminus \{x_0= 0\}$  in this chart.
So we may apply Proposition \ref{selectpr} at the point $(a, b)$.
Hence there exists  an algebraic curve $C_1 \subset \R^{2n}$ of
degree $d_* \le (d+1)^n (d^n +2)^{n-1}$ such that $(a,b) \in
\overline {C_1 \cap Y}$. Now let $C$ be  the projection of $C_1$ on
$\mathbb P^n$.  Note that  the degree of $C$ is less than or equal to $d_*$.  Now we can argue as in the proof of Theorem \ref{luk} to
conclude that there is a real rational arc $x(t)$ of bidegree $D_2
= (d-1)D+1, \, D_1=D$ where $D=(d+1)^n (d^n +2)^{n-1}$, such that
 $\lim_{t\to \infty} F(x(t))=b$.

\end{proof}

\begin{defi}

By a real  rational arc we mean
a function of the form 
$$x(t)=\sum_{- D_2\le i\le D_1} a_it^i, \ t\in
\R^*,$$ where $a_i\in \R^n.$  If $a_i \ne 0$ for $i= D_1,  \, -D_2$ we  say that 
$x(t)$ is of bidegree $(D_1,D_2).$
\end{defi}
We  identify an  arc $x(t)$ with
its coefficients. Clearly the space of all real  rational arcs of bidegree at most  $(D_1,D_2)$ is isomorphic to
$\R^{n(1+D_1+D_2)}$.

\begin{defi}
Let $F=(f_1,\dots ,f_m) :\R^n\to\R^m$ be a generically finite
polynomial mapping which is not proper. Assume that $\deg f_i\le
d$ and $D=(d+1)^n (d^n +2)^{n-1}$  By  the asymptotic variety of real rational
arcs of the mapping $F$ we mean the variety $AV_\R(F)\subset
\R^{n(dD+2)}$ which consists of  those real rational arcs $x(t)$
of bidegree at most  $ (D, (d-1)D+1 )$ such that

a) $F(x(t))=c_0+\sum_{i=1}^\infty c_i/t^i,$

b) $\sum_{i>0} \sum^n_{j=1} a_{ij}^2 =1$, where
$a_i=(a_{i1},\dots ,a_{in}).$

\noindent If we omit  condition b) we get the definition of the
generalized asymptotic variety of real arcs,  $GAV_\R(F).$
\end{defi}

\begin{re}
{\rm Condition b) ensures that the arc $x(t)$ "goes to
infinity".}
\end{re}

As  before we see that $AV_\R(F), GAV_\R(F)$ are algebraic  subsets
of $\R^{n(dD + 2)}$. 
 Recall that we  identify an  arc $x(t)$ with its coefficients $a = (a_{ij})\in \R^{n(2+ \prod^n_{i=1} d_i )}$.
Moreover, for $x(t)\in AV_R(F) $, respectively $x(t)\in GAV_\R(F))$,
we have $F(x(t))= \sum_{i=0}^\infty c_i(a)/t^i.$  Note that again
the function $c_0: AV_\R(F)\to \R^m$ plays an important role.

\begin{pr}
$c_0(AV_\R(F))=S_F(\R).$
\end{pr}

\begin{proof}
Let $x(t)=\sum a_it^i\in AV_\R(F).$ Then
$F(x(t))=c_0(a)+\sum_{i=1}^\infty c_i(a)/t^i$; this implies that
$c_0(a)\in S_F.$ Conversely, let $b\in S_F.$ By Theorem \ref{lukR}
we can find a rational arc $x(t)=\sum_{i=-(d-1)D-1}^D a_it^i$ such
that $\lim_{t\to \infty} F(x(t))=b.$  There exists a $\lambda\in\R$
such that $\sum_{i>0} \sum^n_{j=1} \lambda^{2i} a_{ij}^2 =1.$
Change a parametrization by $t\mapsto\lambda t.$ The new arc
$x'(t):=x(\lambda t)$ belongs to $AV_\R(F)$ and $c_0(x'(t))=b.$
\end{proof}

Now let $f\in \R[x_1,\dots ,x_n]$ be a polynomial. Let us define a
polynomial mapping $\Phi : \R^n \to \R \times \R^{N}$  by
$$\Phi = \left(f,\frac{\partial f}{\partial x_1},\dots,\frac{\partial f}{\partial x_n}, h_{11} ,h_{12},\dots,h_{nn}\right),
$$
where $h_{ij}=x_i\frac{\partial f}{\partial x_j},\, i=1, \dots,n,\,j=1, \dots,n.$

\begin{defi}
Let $\Phi$ be as above. Consider the mapping $c_0: AV_\R(\Phi)\to
\R^{N}$ and the line $L:=\R\times \{ (0,\dots ,0)\}\subset \R\times
\R^N$. By the bifurcation variety of real rational arcs of the
polynomial $f$  we mean the variety
$$BV_\R(f)= \{ x(t)\in AV_\R(\Phi) : x(t)\in c_0^{-1}(L)\}.$$
Similarly we define
$$GBV_\R(f)= \{ x(t)\in GAV_\R(\Phi) : x(t)\in c_0^{-1}(L)\}.$$
\end{defi}

As an immediate consequence of \cite{j-k} we have:

\begin{pr}
Let $K(f)=K_0(f)\cup K_\infty(f)$ denote the set of generalized critical values of a real polynomial $f.$ If we identify the line
$L=\R\times \{ (0,\dots,0)\}\subset \R\times \R^N$ with $\R$, then
we have $c_0(BV_\R(f))=K_\infty(f)$ and
$c_0(GBV_\R(f))=K(f).$
\end{pr}

\section{Real algorithm}

In this section we describe an algorithm to compute the set
$K_\infty(f)$ of asymptotic critical values as well as the set
$K(f)$  of  generalized critical values of a real
polynomial $f\in \R[x_1,\dots,x_n].$ Let $\deg f=d$ and
$D_1=(d+1)^n (d^n +2)^{n-1}, D_2=(d-1)D_1+1.$

\vspace{0.5mm}

{\bf Algorithm for the set $K_\infty(f)$.}

1) Compute equations $g_\alpha$ for the variety $BV_\R(f):$

a) consider an arc $x(t)=\sum_{-D_2}^{D_1} a_i t^i\in
\R^{n(D_1+D_2+1)}$,

b) compute $f(x(t))=\sum c_i(a) t^i,$

c) compute $\frac{\partial f}{\partial x_i}(x(t))=\sum d_{ik}(a)
t^k, i=1,\dots,n$

d) compute $\frac{\partial f}{\partial x_i}(x(t))x_j(t)=\sum
e_{ijk}(a) t^k, i,j=1, \dots,n$,

e) equations for $BV_\R(f)$ are $c_i=0$ for $i>0,$ $d_{ik}=0$ for
$k\ge 0, i=1,\dots,n$, $e_{ijk}=0$ for $k\ge 0, i,j=1,\dots,n$ and
$\sum_{i>0} \sum^n_{j=1} a_{ij}^2 =1$, where
$a_i=(a_{i1},\dots,a_{in}).$

2) Form a polynomial $G=\sum_\alpha g_\alpha^2,$ where $g_\alpha$
are $c_i$ for $i>0,$ or $d_{ik}$ for $k\ge 0, i=1,\dots,n$, or
$e_{ijk}$ for $k\ge 0, i,j=1,\dots,n$ or $\sum_{i>0} \sum^n_{j=1}
a_{ij}^2 - 1.$

3) 
 Using Algebraic Sampling \cite[Theorem 11.61 p. 412]{b-p-r}) we can find a finite set 
 $A\subset \{G=0\}$ which meets every connected component of $\{G=0\}$.

4) $K_\infty(f)=\{c_0(a):\, a\in A\}$.

\vspace{5mm} \noindent If we  replace above the variety $BV_\R(f)$
by the variety $GAV_\R(f)$ we get an  algorithm for computing
$K(f)$. Actually it is enough  to delete   the
equation $\sum_{i>0} \sum^n_{j=1} a_{ij}^2 =1$
in items  1e) and  2).

{\bf Conclusion.} In the paper we have developed  only  the general geometric aspect of our method.
In a forthcoming  work we plan to discuss  effectiveness  of the algorithms and possibly technical improvements.
Recently M. Raibaut  \cite{ra} introduced the motivic bifurcation set of a complex polynomial. It seems that our 
approach can give a more geometric explanation of his work.

      \end{document}